\documentclass[graybox]{svmult}

\usepackage{type1cm}        
\usepackage{makeidx}         
\usepackage{graphicx}        
\usepackage{multicol}        
\usepackage[bottom]{footmisc}

\usepackage{newtxtext}       %
\usepackage[varvw]{newtxmath}       

\newtheorem{hypothesis}[theorem]{Hypothesis}

\makeindex             

\begin{document}

\title*{Convergence of a particle method for gradient flows on the $L^p$-Wasserstein space}
\author{Rong Lei\orcidID{0009-0006-5622-6051}}
\institute{Rong Lei \at Tohoku University, 6-3, Aramaki Aza-Aoba, Aoba-ku, Sendai, Miyagi 980-8578, Japan, \email{lei.rong.e3@tohoku.ac.jp}
}
%
%
\maketitle

\abstract*{}

\abstract{We study the particle method to approximate the gradient flow on the $L^p$-Wasserstein space. This method relies on the discretization of the energy introduced by \cite{CPSW} via nonoverlapping balls centered at the particles and preserves the gradient flow structure at the particle level. We prove the convergence of the discrete gradient flow to the continuum gradient flow on the $L^p$-Wasserstein space over $\mathbb R$, specifically to the doubly nonlinear diffusion equation in one dimension. }

\section{Introduction}
In 1998, Jordan, Kinderlehrer and Otto \cite{JKO} found that entropy is the natural free energy functional in the study of Fokker-Planck equations in non-equilibrium statistical mechanics and proved the Fokker-Planck equation can be viewed as a steepest descent for the associated free energy with respect to the $L^2$-Wasserstein metric by using the minimizing movements method introduced by De Giorgi \cite{DeG}. In 2001, Otto \cite{Otto2001} established the infinite dimensional Riemannian geometric structure on the $L^2$-Wasserstein space and proved that the heat equation and the porous medium equation (or the fast diffusion equation) on Euclidean space are the gradient flows of the Boltzmann entropy and the R\'enyi entropy, respectively, on the $L^2$-Wasserstein space with the infinite dimensional Riemannian metric introduced in \cite{Otto2001}. More generally, the gradient flow of energy $E=\int_{\mathbb R^n}H(\rho(x))dx$ on the $L^2$-Wasserstein space has the form of continuity equation $\partial_t\rho+\nabla\cdot (\rho v) = 0$ with the vector field $v$ given by $v=\nabla\frac{\delta H}{\delta \rho}$, where ${\frac{\delta H}{\delta \rho}}=H^\prime(\rho)$ is the $L^2$-derivative of $H$. See e.g. \cite{AGS, V1, V2}. 

Using De Giorgi's minimizing movements method, a differentiable structure on the underlying space is not required to define the gradient flow, allowing it to be established on general metric spaces (see \cite{AGS}). Consequently, gradient flows can be defined on the $L^p$-Wasserstein space—the space of all probability measures with finite $p$-th moment, equipped with the $W_p$-distance. A notable example of the gradient flow on 
the $L^p$-Wasserstein space is the following Leibenson's equation (see \cite{Leib, IMJ, Vazquez})
\begin{equation}\label{Leibenson}
    \partial_t u=\Delta_q u^\gamma,\quad q>1, \gamma>0,
\end{equation}
which describe the filtration of turbulent compressible fluid through a porous medium, where $\frac{1}{p}+\frac{1}{q}=1$ and the parameter $q$ characterizes the turbulence of the flow while $\gamma-1$ is the index of polytropy of the fluid which determines the relation $PV^{\gamma-1} = const$ between volume $V$ and pressure $P$. Equation \eqref{Leibenson} is a so-called doubly nonlinear diffusion equation. Particularly, when $\gamma=1$, it becomes the $q$-heat equation
\begin{equation}\label{q-heat}
\partial_t  u=\Delta_q u, 
\end{equation}
where $\Delta_q u:=\nabla\cdot\left(|\nabla u|^{q-2}\nabla u\right)$. In general, the gradient flow of $\int_{\mathbb R^n}H(\rho(x))dx$ on $L^p$-Wasserstein space is formally given by 
\begin{equation}\label{GF0}
\begin{cases}
\partial_t\rho+\nabla\cdot\left(\rho|\nabla \varphi|^{q-2}\nabla \varphi\right)=0, \\
\nabla\varphi=-\nabla H^\prime(\rho).
\end{cases}
\end{equation} 
For the rigorous definition, see Section 2. In view of this, the energy corresponding to \eqref{Leibenson} is given by 
\begin{equation}\label{typical energy}
    E(u)=\frac{\gamma}{(\gamma+1-p)(\gamma+2-p)}\int_{\mathbb R^n} u(x)^{\gamma+2-p}dx.
\end{equation}

The discrete approximation to the doubly nonlinear degenerate parabolic equations such as \eqref{Leibenson} is intensively studied. For example, see \cite{EH, ABK} and the reference therein. See also \cite{DEGH} for the gradient discretization method based on some discrete spaces and mappings for nonlinear and nonlocal parabolic equations. In \cite{RMS}, Rossi-Mielke-Savar\'e studied the approximation scheme by time discretization to approach the abstract doubly nonlinear equations in reflexive Banach spaces. On the other hand, in the perspective of gradient flows, Serfaty \cite{Serfaty} proved the $\Gamma$-convergence of gradient flows on metric spaces. By using this result, Carrillo-Patacchini-Sternberg-Wolansky \cite{CPSW} proved the convergence of a particle method to approximate the solutions to the continuum gradient flow on the $L^2$-Wasserstein space, which restrict the continuum gradient flow to the discrete setting of atomic measures, while keeping the gradient flow structure at the discrete level via a suitable approximation of the energy on finite numbers of Dirac masses. The numerical study of this method is also discussed by Carrillo-Huang-Patacchini-Wolansky in \cite{CHPW}. The aim of this paper is to extend this particle method to the approximation of the gradient flows on the $L^p$-Wasserstein space. 

Denote $\Omega^d$ either the closure of a bounded connected domain of $\mathbb R^d$ or $\mathbb R^d$ itself (we simply write $\Omega$ when $d=1$). Let us consider the energy functional $E: \mathcal{P}_p(\Omega^{d})\to(-\infty,+\infty]$ is given by 
\begin{equation}\label{energy}
  E(\mu)=\left\{\begin{aligned}
   &\int_{\Omega^d} H(\rho(x))dx, &\text{ if }\mu=\rho dx\in \mathcal{P}_{p,ac}(\Omega^d),\\
   &+\infty, &\text{ otherwise },
  \end{aligned}
  \right.
\end{equation}
where $\mathcal{P}_{p,ac}(\Omega^d)$ denotes the space of all probability measures on $\Omega^d$ with density functions with respect to the Lebesgue measure $dx$ and with finite $p$-th moment where $p>1$. Moreover, we assume that $H$ satisfies the following hypothesis:
\begin{hypothesis}\label{H1}
Assume $H$ is a proper, convex, non-negative function in $C^\infty\left((0,\infty)\right)\cap C^0\left([0,\infty)\right)$ with superlinear growth at infinity and $H(0)=0$. Moreover, 
\begin{itemize}
  \item $H$ satisfies the doubling condition: there exists a constant $A>0$ such that 
  \begin{equation*}
    H(x+y)\leq A(1+H(x)+H(y)), \quad \forall x,y\in [0,+\infty).
  \end{equation*}
  \item The function $h:x\mapsto x^dH(x^{-d})$ is strictly convex and non-increasing on $(0,+\infty)$. 
\end{itemize}
\end{hypothesis}
Note that the property of superlinear growth at infinity implies that $E$ is lower semi-continuous with respect to the narrow convergence, see \cite[Remark 9.3.8]{AGS}. The assumption that $H(0) = 0$ and $h$ is convex and non-increasing implies that the energy $E$ is displacement convex, which is introduced by McCann \cite{McCann1997}, see also \cite{AGS, V1, V2}. 

\begin{hypothesis}\label{H2}
$H^{\prime\prime}>0$ for all $x\in(0,+\infty)$ and there exists a continuous function $f: (0,+\infty)\to [0,+\infty)$ such that $f(1)=1$ and 
\begin{equation*}
  H^{\prime\prime}(\alpha x)\geq f(\alpha)H^{\prime\prime}(x) \quad\text{ for all }x,\alpha\in (0,+\infty). 
\end{equation*}
\end{hypothesis}
Note that the typical energy \eqref{typical energy} with $\gamma+1-p>0$ satisfies both Hypothesis \ref{H1} and Hypothesis \ref{H2}. 

The rest of this paper is organized as follows. In Section 2, we introduce some notations and give the necessary background to introduce the continuum gradient flow and discrete gradient flow. In Section 3, we state the main result and then prove it. In Section 4, we prove the $\Gamma$-convergence of the discrete energy.

\section{Notations and Preliminaries}
We first recall some basic knowledge about the Wasserstein space and the curve of maximal slope which is used to define the gradient flow. 

Let $(X,d)$ be a Polish space. For $p\geq 1$, the space of all probability measures on $X$ with finite $p$-th moment is denoted by $\mathcal{P}_p(X)$.  
The ($p$-th) Wasserstein distance between two probability measures $\mu, \nu\in \mathcal{P}_p(X)$ is given by 
\begin{equation*}
  W_p(\mu,\nu):=\inf_{\pi}\left\{\int_{X\times X}d(x,y)^p d\pi(x,y)\right\}^{\frac{1}{p}},
\end{equation*}
where
\begin{equation*}
  \pi\in \prod(\mu, \nu)=\left\{\pi\in \mathcal{P}(X\times X):\pi(\cdot,X)=\mu, \pi(X,\cdot)=\nu\right\}.
\end{equation*}
We call $(\mathcal{P}_p(X), W_p)$ the $L^p$-Wasserstein space (or briefly $W_p$ space). Moreover, when the probability measures have smooth densities with respect to some reference measure $dm$, we denote the space by $\mathcal{P}_p^\infty(X,dm)$. Note that $(\mathcal{P}_p(X), W_p)$ is a complete metric space if $X$ is complete. For more properties of the Wasserstein space, one can refer to \cite{AGS, V1, V2}.  

For the case of $p=2$, Otto \cite{Otto2001} introduced the tangent space and the Riemannian metric on the $L^2$-Wasserstein space, which make the $L^2$-Wasserstein space becomes an infinite dimensional Riemannian manifold. In general, for $p\geq 1$, the tangent space of $\mathcal{P}^\infty_p(\mathbb R^n):=\mathcal{P}^\infty_p(\mathbb R^n,dx)$ can be defined as 
\begin{equation*}
  T_{\rho dx}\mathcal{P}_{p}^{\infty}(\mathbb R^n)
  =\left\{s=-\nabla\cdot (\rho|\nabla\phi|^{q-2}\nabla \phi):\phi\in C^{\infty}(\mathbb R^n),\int_{\mathbb R^n}|\nabla\phi(x)|^q \rho(x)dx<\infty \right\},  
\end{equation*} 
where $q$ is the conjugate exponent of $p$, i.e., $\frac{1}{p}+\frac{1}{q}=1$. That is the tangent space of $\mathcal{P}_p(\mathbb R^n)$ can be identified as (see \cite{AGS})
\begin{equation*}
T_{\mu}\mathcal{P}_p(\mathbb R^n):=\overline{\left\{j_q\left(\nabla\phi\right):\phi\in C_c^\infty(\mathbb R^n)\right\}}^{L^p(\mathbb R^n,d\mu)}, 
\end{equation*}
where $j_q: L^q(\mathbb R^n,d\mu)\to L^p(\mathbb R^n,d\mu)$ is defined by 
\begin{equation*}
v\mapsto j_q(v)=\left\{
\begin{aligned}
&|v|^{q-2}v&,\quad \text{ if } v\neq 0,\\
&0&,\quad \text{ if } v= 0.
\end{aligned} \right.
\end{equation*}
Let $E$ be the energy given by \eqref{energy}. If $E$ is regular enough, in view of the definition of the tangent space, the gradient flow on $\mathcal{P}_p(\mathbb R^n)$ can be defined by a solution to 
\begin{equation*}
  j_p(v_t)= -\nabla E(\rho_t). 
\end{equation*}
That is the gradient flow is given by \eqref{GF0}. In general, the gradient flow can be defined with the notion of subdifferential. See Definition \ref{GF1} below. On the other hand, as mentioned in \cite{CPSW}, the gradient flow formulation given in \eqref{GF0} is not the one allows the use of \cite[Theorem 2]{Serfaty}, which is the key tool of our proof. Thus we employ a more weaker notation of gradient flows which only uses a metric structure. The notion replacing gradient flows is then that of “curves of maximal slope”, which was introduced in \cite{DeGMT}. We follow here the self-contained presentation in \cite{AGS}. In the following, let $(X, d)$ be a complete metric space. $\phi$ denotes a proper functional from $X$ to $\mathbb R \cup \{+\infty\}$, i.e. 
\begin{equation*}
D(\phi):=\left\{x\in X: \phi(x)<+\infty\right\}\neq \emptyset, 
\end{equation*}
and $I$ denotes a bounded subinterval of $\mathbb R$. 

\begin{definition}[Absolute continuity]\label{AC}
We say that $v: I \rightarrow X$ is a $p$-absolutely continuous curve if there exists $m \in L^p(I)$ such that
$$
d(v(t), v(\tau)) \leq \int_\tau^t m(s) ds \quad \text { for all } \tau, t \in I \text { with } \tau \leq t.
$$
In this case we denote $v \in A C^p(I, X)$ and $v \in A C(I, X)$ if $p=1$.  
\end{definition} 

\begin{definition}[Metric derivative] 
Let $v: I \rightarrow X$ be a $p$-absolutely continuous curve, then 
$$
\left|v^{\prime}\right|_d(t):=\lim _{\tau \rightarrow t} \frac{d(v(\tau), v(t))}{|\tau-t|}
$$
exists for almost every $t \in I$ and is called the metric derivative of $v$. Moreover, $\left|v^{\prime}\right|_d$ is the smallest admissible function $m$ in Definition \ref{AC}.  
\end{definition}
Note that $v\in AC^p(I,X)$ is equivalent to $|v^\prime|_d\in L^p(I,X)$. 
 
\begin{definition}[Strong upper gradient] 
We call $g: X \rightarrow[0,+\infty]$ a strong upper gradient for $\phi$ if for every $v \in A C(I, X)$, we have that $g \circ v$ is a Borel function and
$$
|\phi(v(t))-\phi(v(\tau))| \leq \int_\tau^t g(v(s))\left|v^{\prime}\right|_d(s) \mathrm{d} s \quad \text { for all } \tau, t \in I \text { with } \tau \leq t. 
$$
\end{definition}
A candidate to be an upper gradient of $\phi$ is its slope: 
\begin{definition}[Local slope] 
We define the local slope of $\phi$ at $v \in D(\phi)$ by
$$
|\partial \phi|(v)=\limsup _{w \rightarrow v} \frac{(\phi(v)-\phi(w))_{+}}{d(v, w)},
$$
where the subscript $+$ denotes the positive part.
\end{definition}

If $\phi$ is $\lambda$-geodesically convex for some $\lambda\in\mathbb R$ and lower semi-continuous, then the local slope $|\partial\phi|$ is a strong upper gradient for $\phi$. See \cite[Corollary 2.4.10]{AGS}. The energy functional \eqref{energy} satisfying Hypothesis \ref{H1} is $0$-geodesically convex on the $L^p$-Wasserstein space. Thus $|\partial E|$ is a strong upper gradient for $E$. 
 
\begin{definition}[Curve of maximal slope] \label{CMS}
Let $g$ be a strong upper gradient for $\phi$. We say that $v\in AC^p(I, X)$ is a $p$-curve of maximal slope for $\phi$ with respect to $g$, if $\phi \circ v$ is almost everywhere equal to a non-increasing function $\varphi$ and
$$
\varphi^{\prime}(t) \leq-\frac{1}{p}\left|v^{\prime}\right|_d(t)^p-\frac{1}{q} g(v(t))^q \quad \text { for almost every } t \in I,
$$
where $q$ is the conjugate exponent of $p$.
\end{definition}

\begin{remark}
 When $v$ is a $p$-curve of maximal slope for a strong upper gradient $g$, we have $g\circ v|v^\prime|_d \in L^1(I)$, $\phi\circ v \in AC(I,\mathbb R\cup\{+\infty\})$, $\phi\circ v(t) = \varphi(t)$ for all $t \in I$, and $|v^\prime|_d(t)^p = g(v(t))^q = -\varphi^\prime(t) = -(\phi\circ v)^\prime(t)$ for almost every $t \in I$ (see \cite[Remark 1.3.3]{AGS}). 
\end{remark}

\subsection{Continuum gradient flow}
Let the energy functional $E$ satisfy Hypothesis \ref{H1}. We now define the continuum gradient flow on $\mathcal{P}_p(\Omega^d)$.

\begin{definition}[Continuum gradient flow] 
We say that $\rho \in AC^p([0,T],\mathcal{P}_p(\Omega^d))$ is a continuum gradient flow solution with initial condition $\rho_0 \in \mathcal{P}_p(\Omega^d)$ if it is a $p$-curve of maximal slope for $E$ with respect to $|\partial E|$ and $\rho(0)=\rho_0$.
\end{definition}
We recall another common way of defining a continuum gradient flow, which involves the notion of subdifferential. For the subdifferential calculus on the $L^p$-Wasserstein space, see \cite[Section 10.3]{AGS}. 

\begin{definition}\label{GF1}
We say that $\mu_t\in AC^p([0,T],\mathcal{P}_p(\Omega^d))$ is a solution to the gradient flow, if there exists a Borel vector field $v_t$ such that $v_t \in T_{\mu_t}\mathcal{P}_p(\Omega^d)$ for $L^1$-a.e. $t > 0$, $\|v_t\|_{L^p(\mu_t)} \in L^p[0, T]$, the continuity equation
\begin{equation*}
  \partial_t\mu_t+\nabla\cdot (v_t\mu_t) =0 \text{ in } \Omega^d\times[0,T]
\end{equation*}
holds in the sense of distribution, and 
\begin{equation*}
  j_p(v_t) \in -\partial E(\mu_t) ,\quad t\in[0,T],
\end{equation*} 
where $\partial E(\mu_t)$ means the subdifferential of $E$ at $\mu_t$.
\end{definition}

The energy functional $E$ satisfying Hypothesis \ref{H1} is regular in the sense of the Definition 10.3.9 of \cite{AGS}. By \cite[Theorem 11.1.3]{AGS}, the definition of curves of maximal slope (Definition \ref{CMS}) coincides with the gradient flow defined by Definition \ref{GF1} on the $L^p$-Wasserstein space. The displacement convexity and lower semi-continuity of $E$ imply the existence of such gradient flows (see \cite[Theorem 11.3.2]{AGS}). Moreover, the tangent vector $v_t$ to $\mu_t$ satisfies the minimal selection principle, i.e., 
\begin{equation*}
    j_p(v_t)=-\partial^\circ E(\mu_t), \text{ for }L^1\text{-a.e. } t>0,
\end{equation*}
where $\partial^\circ E(\mu_t)$ denotes the subset of elements of minimal norm in $\partial E(\mu_t)$, which reduces to a single point since the $L^p(\mu_t)$-norm is strictly convex if $p>1$. 

\subsection{Discrete gradient flow}
Following the particle method used in \cite{CPSW} to approximate the continuum gradient flow, we take any $N$ particles $x_1,\cdots,x_N$ in $\Omega^d$, $N\geq 2$, and we denote $\boldsymbol{x_N}=\left(x_1,\cdots,x_N\right)\in \Omega^{Nd}$. Define the set of empirical measures by
\begin{equation*}
\mathcal{A}_N(\Omega^d)=\left\{\mu\in \mathcal{P}_p(\Omega^d): \exists\left(x_1,\cdots,x_N\right)\in \Omega^{Nd}, \mu=\frac{1}{N}\sum_{i=1}^N\delta_{x_i}\right\}.
\end{equation*}
Let $B_i, i\in\{1,\cdots,N\}$ be the open balls of centre $x_i\in \Omega^d$ with radius $\frac{1}{2}\min\limits_{j\neq i}|x_i-x_j|$, where $|x_i-x_j|$ is the standard Euclidean distance between $x_i$ and $x_j$ on $\mathbb R^d$. Define
\begin{equation*}
  \rho_N=\frac{1}{N}\sum_{i=1}^N\frac{\chi_{B_i}}{|B_i|}\in \mathcal{P}_{ac,p}(\Omega^d),
\end{equation*}
where $|B_i|$ is the volume of $B_i$ with respect to the Lebesgue measure on $\mathbb R^d$. 

\begin{definition}[Discrete energy] 
We define the discrete energy $E_N: \mathcal{A}_N(\Omega^d)\to\mathbb R$ for all $\mu_N\in \mathcal{A}_N(\Omega^d)$ with particles $\boldsymbol{x_N}\in\Omega^{Nd}$ by 
\begin{equation}\label{dis energy}
E_N(\mu_N)=E(\rho_N)=\sum_{i=1}^N|B_i|H\left(\frac{1}{N|B_i|}\right).
\end{equation}
We define the discrete energy equivalently as a function of $\boldsymbol{x_N}\in \Omega^{Nd}$ by
\begin{equation*}
  \widetilde{E}_N(\boldsymbol{x_N})=E_N(\mu_N). 
\end{equation*}
\end{definition}

Now we introduce the weighted $p$-norm on $\mathbb R^{Nd}$. That is, for any $\boldsymbol{x_N}=(x_i,\cdots,x_N)\in\mathbb R^{Nd}$, $x_i=(x_i^1,\cdots,x_i^d)$ for $i\in\{1,\cdots,N\}$, 
\begin{equation*}
\|\boldsymbol{x_N}\|_{w,p}:=\left\{\frac{1}{N}\sum_{i=1}^N\sum_{j=1}^d |x_i^j|^p\right\}^{1/p}.
\end{equation*}
Then $(\mathbb R^{Nd},\|\cdot\|_{w,p})$ becomes a Banach space with the dual space $(\mathbb R^{Nd},\|\cdot\|_{w,q})$, where $\frac{1}{p}+\frac{1}{q}=1$. Moreover, the pair of $\boldsymbol{x_N}\in(\mathbb R^{Nd},\|\cdot\|_{w,p})$ and $\boldsymbol{y_N}\in(\mathbb R^{Nd},\|\cdot\|_{w,q})$ is given by 
\begin{equation*}
    \left( \boldsymbol{x_N},\boldsymbol{y_N}\right)_w=\frac{1}{N}\sum_{i=1}^{N}\sum_{j=1}^dx_i^jy_i^j.
\end{equation*}

Now we define the discrete gradient flow. As mentioned in \cite{CPSW}, to obtain the well-posedness of the discrete gradient flow, we restrict the framework to the case of $d=1$. In this case, the discrete energy $\widetilde{E}_N$ is convex on $\Omega^N$, which makes sure that the local slope is a strong upper gradient. 

By convention, in the rest of the paper, whenever particles $\boldsymbol{x_N}\in\Omega^N$ are considered, they are assumed to be distinct and sorted increasingly, i.e., $x_{i+1} > x_i$ for all $i\in\{1,\cdots, N-1\}$. We also assume the same boundary condition as in \cite{CPSW}, which construct two fictitious particles $x_0$ and $x_{N+1}$ to make sure the real particles stay in $\Omega$. We denote $\Delta x_i=x_i-x_{i-1}$ for $i\in\{1,\cdots, N+1\}$. 

\begin{definition}[Discrete gradient flow]
We say that $\mu_N\in AC^p([0,T],\mathcal{A}_N(\Omega))$ is a discrete gradient flow solution with initial condition $\mu_N^0 \in\mathcal{A}_N(\Omega)$, if it is a $p$-curve of maximal slope for $E_N$ with respect to $|\partial E_N |$, and if $\mu_N(0) =\mu_N^0$.
\end{definition}
By \cite[Proposition 2.13]{CPSW}, it is equivalent to define the discrete gradient flow as a solution to
\begin{equation}\label{dis GF}
  j_q \boldsymbol{x}^{\prime}(t)\in \partial_w \widetilde{E}_N(\boldsymbol{x}(t)), \quad \forall t\in [0,T],
\end{equation}
where $\boldsymbol{x} : [0, T ]\to \mathbb R^{N}$ and $\partial_w$ stands for the subdifferential of $\widetilde{E}_N$:
\begin{equation*}
  \partial_w \widetilde{E}_N(\boldsymbol{x})=\left\{\boldsymbol y\in\mathbb R^{N} \Big| \forall \boldsymbol z\in \mathbb R^{N}, E_N(\boldsymbol z)-E_N(\boldsymbol x)\geq \left(\boldsymbol y, \boldsymbol z- \boldsymbol x\right)_w\right\}.
\end{equation*}
Moreover, by \cite[Proposition 1.4.4, Corollary 2.4.12]{AGS}, we have the following 
\begin{proposition}
There exists a solution to the discrete gradient flow inclusion \eqref{dis GF}. Furthermore, any solution $\boldsymbol{x_N}$ satisfies $j_q \boldsymbol{x}^{\prime}(t)=\partial_w^0 \widetilde{E}_N(\boldsymbol{x}(t))$ for almost every $t\in[0,T]$.    
\end{proposition}

\section{main results}
Before stating the main result, we recall some definitions introduced in \cite{CPSW}.
\begin{definition}[Smooth set]
We define the subset $\mathcal{G}(\Omega)$ of $\mathcal{P}_{ac,p}(\Omega)$ as follows. We write $\rho \in \mathcal{G}(\Omega)$ if there exists $r > 0$ such that all the items below hold:
\begin{itemize}
    \item $\mathrm{supp}\rho=[-r,r]$; $\rho|_{\mathrm{supp}\rho}\in C^1(\mathrm{supp}\rho)$; $\min_{\mathrm{supp}\rho}\rho>0$;
    \item if $\Omega=[-l,l]$, then $r=l$.
\end{itemize}
\end{definition}

\begin{definition}[recovery sequence and well-preparedness]\label{well-pre}
Let $\rho\in\mathcal{P}_{p}(\Omega)$. Any $(\mu_N)_{N\geq2}$ with $\mu_N \in \mathcal{A}_N(\Omega)$ for all $N\geq 2$ such that $\mu_N\rightharpoonup\rho$ narrowly as $N\to\infty$ and $\limsup_{N\to\infty} E_N(\mu_N) \leq E(\rho)$ is said to be a recovery sequence for $\rho$. Let $(\boldsymbol{x_N})_{N\geq 2}$ be the particles of $(\mu_N)_{N\geq2}$. We say that $(\mu_N)_{N\geq2}$ is well-prepared for $\rho$ if it is a recovery sequence for $\rho$ and there exist $a_1,a_2 > 0$ such that $a_1/N \leq \Delta x_i \leq a_2/N$ for all $i\in\{2,\cdots,N\}$ and all $N\geq 2$; if $\rho\in \mathcal{G}(\Omega)$, we moreover require $x_N =-x_1 = r$.  
\end{definition}

Now we state our main result. 
\begin{theorem}[Main theorem]\label{MT}
Assume $H$ satisfies Hypothesis \ref{H1}. Suppose $\mu_N \in AC^p\left([0, T], \mathcal{A}_{N}(\Omega)\right)$, with particles $\boldsymbol{x_N} \in AC^p\left([0, T], \Omega^N\right)$, is a discrete gradient flow solution with initial condition $\mu_N^0 \in \mathcal{A}_{N}(\Omega)$, with particles $\boldsymbol{x_N^0} \in \Omega^N$. Let $\rho\in \mathcal{G}(\Omega)$, and assume that $(\mu_0)$ is well-prepared for $\rho$ according to Definition \ref{well-pre}. Then $\left(\mu_N(t)\right)_{N\geq 2}$ is tight and there exists a subsequence $(\mu_{N_k}(t))_{N\geq 2}$ and a probability measure $\rho \in AC^p\left([0, T], \mathcal{P}_p(\Omega)\right)$ such that
\begin{equation*}
  \mu_{N_k}(t) \rightharpoonup \rho(t) \text{ narrowly }
\end{equation*}
as $k \rightarrow \infty$ for all $t \in[0, T]$. Moreover, if $\Omega=[-\ell, \ell]$ and $H$ satisfies Hypothesis \ref{H2}, then $\rho$ is a continuum gradient flow and it holds\footnote{When we write the metric derivative with respect to the $W_p$-distance, we omit the subscript. }
\begin{equation}\label{conclusion}
\begin{cases}
\lim\limits_{k \rightarrow \infty}\left|\mu_{N_k}^{\prime}\right|=\left|\rho^{\prime}\right| \quad &\text { in } L^p([0, T]), \\
\lim\limits_{k \rightarrow \infty} E_{N_k}\left(\mu_{N_k}(t)\right)=E(\rho(t)) \quad &\text { for all } t \in[0, T] ,\\
\lim\limits_{k \rightarrow \infty}\left|\partial E_{N_k}\left(\mu_{N_k}\right)\right|=|\partial E(\rho)| \quad &\text { in } L^p([0, T]).
\end{cases} 
\end{equation}
In particular, assuming $\Omega=[-l,l]$. If $p=2$ or $E$ is given by \eqref{typical energy} with $\gamma+1-p>0$, then $\mu_N(t)$ narrowly converges to $\rho(t)$ for all $t\in[0,T]$ and \eqref{conclusion} holds for whole sequence $\mu_N(t)$ and $E_N$. 
\end{theorem} 

\begin{remark}
The first part of this theorem is the tightness of $\left(\mu_t^N\right)_{N\geq 2}$ which will be proved in Proposition \ref{tightness prop}. Then by Prohorov’s theorem, $\left(\mu_t^N\right)_{N\geq 2}$ is narrowly sequentially compact. Using Theorem \ref{Serfaty thm}, we prove that the sequential limit $\rho$ is a solution to the continuum gradient flow. Moreover, if we know the uniqueness of the solution to the continuum gradient flow, which is established for the case of $p=2$ (see \cite[Theorem 11.1.4]{AGS}), $q$-heat equation \eqref{q-heat} (see \cite[Chapter 11]{Vazquez} and \cite{Kell}) and the Leibenson's equation \eqref{Leibenson} (see \cite{IMJ}), then we can obtain that the whole sequence of $\mu_N$ narrowly converges to $\rho$ and \eqref{conclusion} holds for whole sequence $\mu_N(t)$ and $E_N$.  
\end{remark}

\begin{remark}
In Theorem \ref{MT}, it can actually be proved that the convergence of $\mu_{N_k}(t)$ to $\rho(t)$ is stronger than narrowly convergence. Indeed, we will prove $\mu_N(t)$ has bounded $p$-th moment uniformly in $N$ and $t\in[0,T]$. See the proof of Proposition \ref{tightness prop}. Thus $\mu_{N_k}(t)$ converges to $\rho(t)$ in $W_r$-distance for all $1< r\leq p$ and $t\in[0,T]$, see e.g. \cite[Proposition 7.1.5]{AGS}.    
\end{remark}

To prove this theorem, we use the following theorem proved by Serfaty in \cite{Serfaty}, which is the $\Gamma$-convergence of gradient flow on metric spaces. 
\begin{theorem}[\cite{Serfaty}]\label{Serfaty thm}
Let $\mu_N\in AC^p([0,T],\mathcal{A}_N(\Omega))$ be a discrete gradient flow. Assume that $\mu_N(t)\rightharpoonup \rho(t)$ narrowly as $N \to\infty$ for all $t \in [0,T]$ for some $\rho\in AC^p([0,T],\mathcal{P}_p(\Omega))$. Furthermore, suppose that $(\mu_N(0))_{N\geq2}$ is a recovery sequence for $\rho(0)$, and that the following conditions hold for all $t\in [0, T ]$.
\begin{itemize}
  \item [(C1)] \qquad$\liminf\limits_{N \rightarrow \infty}\int_0^t\left|\mu_N^{\prime}\right|(s)^2ds\geq \int_0^t\left|\rho^{\prime}\right|(s)^2ds$.
  \item [(C2)] \qquad$\liminf\limits_{N \rightarrow \infty} E_N\left(\mu_N(t)\right)\geq E(\rho(t))$.
  \item [(C3)]\qquad$\liminf\limits_{N \rightarrow \infty}\left|\partial E_N\right|\left(\mu_N(t)\right)\geq |\partial E|(\rho(t))$. 
\end{itemize}
Then $\rho$ is a continuum gradient flow and \eqref{conclusion} holds for $E_N$ and $\mu_N$.  
\end{theorem}

\subsection{Tightness and condition on the metric derivatives}
Now we prove the first part of the Main theorem. 
\begin{proposition}\label{tightness prop}
Let $\Omega=\mathbb R$ or $\Omega=[-l,l]$ for fixed $l\in\mathbb R$. Assume $H$ satisfies Hypothesis \ref{H1}. Let $(\mu_N)_{N\geq2}$ be as in Theorem \ref{MT}. Then $(\mu_N)_{N\geq2}$ is tight in $AC^p([0, T ], \mathcal{P}_p(\Omega))$. Moreover, there exists a subsequence $\mu_{N_k}$ narrowly converges to some $\rho\in AC^p([0, T ], \mathcal{P}_p(\Omega))$ as $k \to\infty$ for all $t \in [0, T ]$. Furthermore, $(C1)$ holds for this subsequence. 
\end{proposition}

\begin{proof}
For any $\sigma\in\mathcal{P}_p(\Omega)$, let $\mu_t$ be an absolutely continuous curve on $\mathcal{P}_p(\Omega)$ with velocity $v_t$. Then the differentiability of $W_p$ gives 
\begin{equation*}
\frac{1}{p}\frac{d}{dt}W_p^p(\mu_t, \sigma)
=\int_{\Omega^2}\langle v_t(x_1),j_p(x_1-x_2)\rangle d\gamma_t(x_1.x_2),\quad L^1 \text{-a.e. } t \in (0, +\infty),
\end{equation*}
where $\gamma_t\in \Gamma_o(\mu_t, \sigma)$ is an optimal transport plan form $\mu_t$ to $\sigma$. Applying the H\"older's inequality, it holds 
\begin{eqnarray*}
\frac{1}{p}\frac{d}{dt}W_p^p(\mu_t, \sigma)
&\leq &\|v_t\|_{L^p(\Omega,\mu_t)}\left(\int_{\Omega^2}\left|j_p(x_1-x_2)\right|^qd\gamma_t(x_1,x_2)\right)^{1\over q}\\
&=&\|v_t\|_{L^p(\Omega,\mu_t)}\left(\int_{\Omega^2}\left|x_1-x_2\right|^p d\gamma_t(x_1,x_2)\right)^{\frac{1}{p}\cdot\frac{p}{q}}\\
&=&\|v_t\|_{L^p(\Omega,\mu_t)}\left(W_p(\mu_t,\sigma)\right)^{p-1},
\end{eqnarray*}
Let $\mu_t$ be a $p$-curve of maximal slope for $E$ with respect to $|\partial E|$. Then 
\begin{equation*}
    |\mu_t^\prime|=\|v_t\|_{L^p(\Omega,\mu_t)} \quad \text{ for } L\text{-a.e. }t\in [0,T].
 \end{equation*} 
By \cite[Remark 2.4.17]{AGS}, we have the following estimate
\begin{equation*}
    t|\partial E|^q(\mu_t)\leq E(\mu_0)-\inf_{\mu\in D(E)} E(\mu)\leq E(\mu_0).
\end{equation*}
Since $|\partial E|$ is a strong upper gradient of $E$, it holds 
\begin{equation*}
    |\partial E|^q(\mu_t)=|\mu_t^\prime|^p.
\end{equation*}
Thus we have 
\begin{equation*}
\frac{1}{p}\frac{d}{dt}W_p^p(\mu_t, \sigma)\leq \left(\frac{E(\mu_0)}{t}\right)^{\frac{1}{p}}\left(W_p(\mu_t,\sigma)\right)^{p-1}.
\end{equation*}
Then we can derive  
\begin{equation*}
W_p(\mu_t, \sigma)\leq qE(\mu_0)^{1\over p}\left(t^{1\over q}-t_0^{1\over q}\right)+W_p(\mu_{t_0},\sigma), \quad \forall 0\leq t_0\leq t\leq T.
\end{equation*}
In particular, we choose $t_0=0$ and $\sigma=\mu_0$, then 
\begin{equation*}
W_p(\mu_t, \mu_0)\leq qE(\mu_0)^{1\over p}t^{1\over q}\leq qE(\mu_0)^{1\over p}T^{1\over q}.
\end{equation*}
Let $\mu_t^N$ be a gradient flow of $E_N$ with initial value $\mu_0^N$. Then 
\begin{equation*}
W_p(\mu_t^N, \mu_0^N)\leq qE_N(\mu_0^N)^{1\over p}T^{1\over q}.
\end{equation*}
Assuming there exists a constant $e_0$ such that $E_N(\mu_0^N)\leq e_0$, then 
\begin{equation*}
W_p(\mu_t^N, \mu_0^N)\leq qe_0^{1\over p}T^{1\over q}.
\end{equation*}
Note that 
\begin{eqnarray*}
M_p(\mu_t^N)&=&W_p^p(\mu_t^N,\delta_0)\leq \left(W_p(\mu_t^N,\mu_0^N)+W_p(\mu_0^N,\delta_0)\right)^p\\
&\leq& 2^{p-1} W_p^p(\mu_t^N,\mu_0^N)+2^{p-1} W_p^p(\mu_0^N,\delta_0)\\
&\leq &2^{p-1}q^p e_0 T^{p\over q}+2^{p-1} M_p(\mu_0^N),
\end{eqnarray*}
where $M_p$ denotes the $p$-th moment. Assume there exists a constant $m_0$ such that $M_p(\mu_0^N)\leq m_0$, and we have 
\begin{equation}\label{unif bd of p-moment}
M_p(\mu_t^N)\leq 2^{p-1}q^p e_0 T^{p\over q}+2^{p-1} m_0.   
\end{equation}
That is $\mu_t^N$ have bounded $p$-th moment ($p>1$) uniformly in $N$ and $t\in[0,T]$, then the Chebyshev's inequality gives the uniformly integrability of $\mu_t^N$, which imply the tightness of $\mu_t^N$. By Prohorov's theorem, there exist a subsequence $\mu_{N_k}(t)$ and $\rho\in C([0,T], \mathcal{P}_p(\Omega))$ such that $\mu_{N_k}(t)$ narrowly converges to $\rho(t)$ as $k\to\infty$ for all $t\in[0,T]$. 

Now we show that $\rho$ is actually in $AC^p([0,T],\mathcal{P}_p(\Omega))$. By \cite[Theorem 11.3.2]{AGS},  
\begin{equation*}
    \int_0^t |\mu_N^\prime|^p(s)ds=E_N(\mu_N^0)-E_N(\mu_N(t))\leq e_0, \quad t\in[0,T], 
\end{equation*}
which means the metric derivative $|\mu_N^\prime|$ is bounded in $L^p([0,t])$, thus it is $L^p$-weakly convergent to some $v\in L^p([0,t])$ up to a subsequence (still denoted by $|\mu_N^\prime|$). In particular, we can choose the test function by the characteristic function $\chi_{[0,T]}\in L^q([0,T])$, then we have 
\begin{equation}\label{4.9}
\lim_{N\to \infty}\int_{t_0}^{t_1}|\mu_N^\prime|(s)ds=\int_{t_0}^{t_1}v(s)ds \quad \text{ for all } 0\leq t_0\leq t_1\leq T.
\end{equation}
Note that $\mu_N$ is $p$-absolutely continuous, by definition of the metric derivative, 
\begin{equation*}
    W_p(\mu_N(t_0),\mu_N(t_1))\leq \int_{t_0}^{t_1}|\mu_N^\prime|(s)ds.
\end{equation*}
Then, by \eqref{4.9} and the narrow lower semi-continuity of $W_p$ (see \cite[Proposition 7.1.3]{AGS}), 
\begin{equation*}
W_p(\rho(t_0),\rho(t_1))\leq \int_{t_0}^{t_1}v(s)ds.    
\end{equation*}
Therefore $\rho\in AC^p([0,T],\mathcal{P}_p(\Omega))$. Moreover, $|\rho^\prime|(s)\leq v(s)$ for almost every $s\in[0,T]$. By the weak lower semi-continuity of the $L^p$-norm, this gives
\begin{equation*}
\liminf_{N\to\infty}\int_0^t |\mu_N^\prime|^p(s)ds=\lim_{N\to\infty}\int_0^t |\mu_N^\prime|^p(s)ds\geq \int_0^t v(s)^p ds\geq \int_0^t |\rho^\prime|^p(s)ds, 
\end{equation*}
which is $(C1)$. 
\end{proof}

\subsection{Condition on the energy}
Now we verify the “lower semi-continuity” conditions on the energies and the slopes of the energies. 
\begin{proposition}\label{C2 holds}
Let $H$ satisfy Hypothesis \ref{H1} and let $\left(\mu_N\right)_{N\geq 2}$ and $\rho$ be as in Theorem \ref{MT}. Then $(C2)$ holds. 
\end{proposition}
\begin{proof}
From now on, we denote the subsequence in Theorem \ref{MT} by $(\mu_N)_{N\geq 2}$. By definition, $E_N(\mu_N)=E(\rho_N)$. Since $E$ is narrowly lower semi-continuous, we need to prove $\rho_N$ narrowly converges to $\rho$ as $N\to\infty$. By tightness, $\mu_N\rightharpoonup\rho$ narrowly, thus we only need to prove $\rho_N\rightharpoonup \mu_N$ narrowly as $N\to\infty$. By the density of Lipschitz function in $C_b(\mathbb R)$, we only test against $\varphi\in C_b(\mathbb R)$ with Lipschitz constant $L>0$. Compute 
\begin{equation*}
\begin{aligned}
&\left|\int_{\mathbb R}\varphi(x)\rho_N(x)dx-\int_{\mathbb R}\varphi(x)d\mu_N(x)\right|=\left|\sum_{i=1}^N\frac{1}{N}\int_{B_i}\frac{\varphi(x)}{|B_i|}dx-\frac{1}{N}\sum_{i=1}^N\varphi(x_i)\right|\\
\leq& \frac{1}{N}\sum_{i=1}^N\frac{1}{|B_i|}\int_{B_i}\left|\varphi(x)-\varphi(x_i)\right|dx\leq \frac{L}{N}\sum_{i=1}^N\frac{1}{|B_i|}\int_{B_i}|x-x_i|dx\leq \frac{L}{4N}\sum_{i=1}^N|B_i|\\
\leq&\frac{L}{N}\sum_{i=2}^N \Delta x_i+\frac{L}{N}\Delta x_2\leq \frac{2^{\frac{p-1}{p}}L}{N}\left(\left(x_1^p+x_N^p\right)^{1/p}+\left(x_1^p+x_2^p\right)^{1/p}\right)\\
\leq& 2^{\frac{p-1}{p}}L\left(\frac{M_p(\mu_N)}{N^{p-1}}\right)^{1/p}.
\end{aligned}
\end{equation*} 
By \eqref{unif bd of p-moment}, we have 
\begin{equation*}
\left|\int_{\mathbb R}\varphi(x)\rho_N(x)dx-\int_{\mathbb R}\varphi(x)d\mu_N(x)\right|\to 0, \text{ as } N\to\infty.
\end{equation*}
\end{proof}

The last step is to verify the condition $(C3)$ on the local slopes. In this section, we denote $g:= |\partial E|$ and $g_N:= |\partial E_N |$, and we take $\Omega= [-l, l]$. In this case, the local slope $g$ of $E$ is given in the lemma below.  
\begin{lemma}\cite[Theorem 10.4.6]{AGS}
Let $H$ satisfies the Hypothesis \ref{H1}. Then the local slope of $E$ is given by 
\begin{equation*}
g(\rho)=\left(I_p(\rho)\right)^{1\over p},
\end{equation*}
where $I_p(\rho)$ is the generalized Fisher information on $\mathcal{P}_p([-l,l])$ which is defined by 
\begin{equation*}
I_p(\rho)=\left\{
\begin{aligned}
&\int_{-l}^{l}|H^{\prime\prime}(x)\rho^\prime(x)|^p\rho(x)dx,\quad 
    \begin{aligned}
    &\text{ if } \rho\in \mathcal{P}_{ac,p}([-l,l])\\
    &\text{ and }L_H(\rho(\cdot))\in W^{1,1}([-l.l]),  
    \end{aligned}\\
& +\infty,\quad \text{ otherwise, } 
\end{aligned}
\right.     
\end{equation*} 
where $L_H(\rho)=\rho H^\prime(\rho)-H(\rho)$. 
\end{lemma}

Now we verify the condition $(C3)$. 
\begin{proposition}\label{prop condition 3}
Let $H$ satisfies Hypothesis \ref{H1} and Hypothesis \ref{H2}. Let $\rho$ be as in Theorem \ref{MT}. Then $(C3)$ holds. That is 
\begin{equation*}
\liminf_{N\to\infty}g_N(\mu_N(t))\geq g(\rho(t)), \quad \forall t\in[0,T]. 
\end{equation*} 
\end{proposition}

First, we compute explicitly the local slope of $E_N$. Denote $g_N(\mu_N)=|\partial E_N|(\mu_N)$ for $\mu_N\in \mathcal{A}_N([-l,l])$. By the definition of local slope, we have 
\begin{equation*}
g_N(\mu_N)=\|\partial^0_w\widetilde{E}_N(\boldsymbol{x_N})\|_{w,p}.
\end{equation*}
We use the notation and strategy in \cite{CPSW} to describe whether the closest neighbour to that particle is to the right and to characterize $\partial_w\widetilde{E}_N$. Given $\boldsymbol{x_N}\in [-l,l]^N$, we write $(\lambda^-,\lambda,\lambda^+)\in \Lambda(\boldsymbol{x_N})$ if 
{\small\begin{equation*}
\lambda_i^{-}\left\{\begin{array} { l l } 
{ = 0 } & { \text { if } \Delta x _ { i } > \Delta x _ { i - 1 } , } \\
{ \in [ 0 , 1 ] } & { \text { if } \Delta x _ { i } = \Delta x _ { i - 1 } ,} \\
{ = 1 } & { \text { if } \Delta x _ { i } < \Delta x _ { i - 1 } , }
\end{array} ~ \lambda _ { i }\left\{\begin{array} { l l } 
{ = 0 } & { \text { if } \Delta x _ { i + 1 } > \Delta x _ { i } , } \\
{ \in [ 0 , 1 ] } & { \text { if } \Delta x _ { i + 1 } = \Delta x _ { i } , } \\
{ = 1 } & { \text { if } \Delta x _ { i + 1 } < \Delta x _ { i } , }
\end{array} ~ \lambda _ { i } ^ { + }\left\{\begin{array}{ll}
=0 & \text { if } \Delta x_{i+2}>\Delta x_{i+1} \\
\in[0,1] & \text { if } \Delta x_{i+2}=\Delta x_{i+1} \\
=1 & \text { if } \Delta x_{i+2}<\Delta x_{i+1}
\end{array}\right.\right.\right.
\end{equation*}}
for all $i \in\{1, \ldots, N\}$, with the convention that $\Delta x_1>\Delta x_0$ and $\Delta x_{N+1}>\Delta x_{N+2}$. 
\begin{lemma}
Take $\boldsymbol{x_N}\in [-l,l]^N$. We have 
\begin{equation*}
\partial_w \widetilde{E}_N(\boldsymbol{x_N})=\left\{z\in \mathbb R^N \Big|\quad
\begin{aligned}
 &\exists (\lambda^-,\lambda,\lambda^+)\in \Lambda(\boldsymbol{x_N}), \text{ for }1\leq i\leq N, \\
 &z_i=(\lambda_i-\lambda_i^{+}+1)\psi_{i+1}-(\lambda_i^--\lambda_i+1)\psi_{i}
\end{aligned}
\right\}, 
\end{equation*} 
where $\psi_i:=-Nh^{\prime}(N\Delta x_i)$ for all $i\in \left\{1,\cdots,N\right\}$.
\end{lemma}

\begin{proof}
Denote $r_i=\min\{\Delta x_i,\Delta x_{i+1}\}$ for $i\in\{1,\cdots,N\}$. We have 
\begin{equation*}
\begin{aligned}
\widetilde{E}_N\left(\boldsymbol{x}_{\boldsymbol{N}}\right)
&=\frac{1}{N} \sum_{i=1}^N h\left(N r_i\right)=\frac{1}{N} \sum_{i=1}^N \max \left[h\left(N \Delta x_i\right), h\left(N \Delta x_{i+1}\right)\right]\\
&=\frac{1}{N}\Big[\max\left\{h(N\Delta x_{i-1}),h(N\Delta x_{i})\right\}+\max\left\{h(N\Delta x_{i}),h(N\Delta x_{i+1})\right\}\\
&\qquad\quad +\max\left\{h(N\Delta x_{i+1}),h(N\Delta x_{i+2})\right\}+\sum_{k\notin\{i-1,i,i+1\}}h\left(N r_i\right)\Big],
\end{aligned}   
\end{equation*}
where the function $h$ is a smooth convex and non-increasing function on $(0, \infty)$. Moreover, one can check that $x\mapsto N r_i(x)$ is Lipschitz continuous around $x$. 
Thus we can use the chain rule of subdifferential (see \cite[Theorem 1.110]{Mor}) and obtain 
{\small\begin{equation*}
\begin{aligned}
\partial^i\widetilde{E}_N\left(\boldsymbol{x}_{\boldsymbol{N}}\right)&=\lambda_i^-h^\prime(N\Delta x_i)+(-\lambda_i+1)h^\prime(N\Delta x_i)-\lambda_ih^\prime(N\Delta x_{i+1})+(\lambda_i^+-1)h^\prime(N\Delta x_{i+1})\\
&=\frac{1}{N}\left(\lambda_i-\lambda_i^{+}+1\right)\psi_i-\frac{1}{N}\left(\lambda_i^{-}-\lambda_i+1\right)\psi_{i+1},
\end{aligned}
\end{equation*}}
which gives the conclusion. 

\end{proof}

By the same method in \cite{CPSW} and going through each case of the triplets $(\lambda^-,\lambda,\lambda^+)\in \Lambda(\boldsymbol{x_N})$, we have the following Lemma \ref{lem 6.4}, Lemma \ref{lem 6.5} and Lemma \ref{lem 6.6}. We assume the same boundary condition as in \cite{CPSW}. Since the proofs of these lemmas can be modified using the same strategy as in \cite{CPSW}, we omit details here. 
\begin{lemma}\label{lem 6.4}
Let $z=(z_1,\cdots,z_N)\in \partial_w^0 E_N(\boldsymbol{x}_{\boldsymbol{N}})$. Then $|z_i|\geq |\psi_i-\psi_{i+1}|$.    
\end{lemma}

\begin{lemma}\label{lem 6.5}
Let $\boldsymbol{x_N} \in [-l,l]^N$ be as assumed in Theorem \ref{MT}. Then $x_N(t)=-x_1(t)=l$ for all $t \in [0,T]$. 
\end{lemma}

\begin{lemma}\label{lem 6.6}
Let $\boldsymbol{x_N} \in [-l,l]^N$ be as assumed in Theorem \ref{MT}. Then, for all $t \in [0,T]$,
\begin{equation*}
a_1N^{-1}\leq\Delta x_i(t)\leq a_2N^{-1} \text{ for all }i\in\{2,\cdots,N\}.
\end{equation*}
The constants $a_1$ and $a_2$ are those of Definition \ref{well-pre} for the well-prepared set $\boldsymbol{x_N^0}$ for $\rho_0$.   
\end{lemma}

The following lemma is the key to prove the convergence in $(C3)$. 
\begin{lemma}\label{key lemma}
Suppose that $\liminf\limits_{N\to\infty}g_N(\mu_N(t))$ is finite for all $t\in[0,T]$. Then
\begin{equation}\label{convergence of step}
\max_{i\in\{2,\cdots,N-1\}}\left|\frac{\Delta x_{i+1}}{\Delta x_i}-1\right|\xrightarrow{N\to\infty}0, \qquad \forall t\in[0,T]. 
\end{equation}  
\end{lemma}

\begin{proof}
Let $z=(z_1,\cdots,z_N)\in \partial_w^0 E_N(\boldsymbol{x}_{\boldsymbol{N}})$. Then by Lemma \ref{lem 6.4}, 
\begin{equation*}
g_N(\mu_N)^p=\frac{1}{N}\sum_{i=1}^N |z_i|^p\geq \frac{1}{N}\sum_{i=1}^N \left|\psi_i-\psi_{i+1}\right|^p.
\end{equation*}
By Lemma \ref{lem 6.5}, it holds $\Delta x_1=\Delta x_2$ and $\Delta x_N=\Delta x_{N+1}$. Therefore $\psi_1 = \psi_2$ and $\psi_N = \psi_{N+1}$. Noticing that $h$ is smooth, it follows that
\begin{equation*}
\begin{aligned}
g_N(\mu_N)^p&\geq \frac{1}{N}\sum_{i=1}^N \left|\psi_i-\psi_{i+1}\right|^p=\frac{1}{N}\sum_{i=1}^N N^p\left|h(N\Delta x_i)-h(N\Delta x_{i+1})\right|^p\\
&=N^{p-1}\sum_{i=1}^N \left|h^\prime(\xi)\right|^p N^p\left|\Delta x_i-\Delta x_{i+1}\right|^p,
\end{aligned}
\end{equation*}
where $\xi=\theta(N\Delta x_i)+(1-\theta)(N\Delta x_{i+1})$ for some $\theta\in[0,1]$. By Lemma \ref{lem 6.6}, it holds $a_1\leq \xi\leq a_2$ for all $t\in[0,T]$. Therefore $h^\prime(\xi)$ is finite and independent of $N$. Thus $\liminf\limits_{N\to\infty}g_N(\mu_N(t))<\infty$ implies $\eqref{convergence of step}$. 
\end{proof}

\begin{proof}[Proof of Proposition \ref{prop condition 3}]
We omit the time dependence. First we define the interpolation $\widetilde{\rho}_N$ to approximate $\rho$. Using the same method and notations as in \cite{CPSW}, we introduce the function $\psi:(0,\infty)\to[0,\infty)$ by 
\begin{equation*}
\psi(x)=-h^\prime(x), \quad \forall x\in (0,\infty).
\end{equation*} 
Clearly $\psi_i=N\psi(N\Delta x_i)$. For $i\in\{1,\cdots,N-1\}$, we introduce the monotone function $p_i:[x_i,x_{i+1}]\to (0,\infty)$ by 
\begin{equation*}
    p_{i}(x)=\frac{1}{N\Delta x_{i+1}}\left[(x-x_i)\psi_{i+1}-(x_{i+1}-x_i)\psi_i\right] \quad \text{for } x\in[x_i,x_{i+1}].
\end{equation*}
Obviously $p_i(x_i)=\frac{\psi_i}{N}$, $p_i(x_{i+1})=\frac{\psi_{i+1}}{N}$. Since $h^\prime$ is strictly increasing, $\psi$ is strictly decreasing and therefore invertible\footnote{Since the function $\psi$ in \cite{CPSW} is also required to be invertible, the strictly convexity of function $h$ in \cite{CPSW} is also necessary. }. Define 
\begin{equation}\label{tilde rho_N}
\widetilde{\rho}_N(x):=\frac{1/m_N}{\psi^{-1}(p_i(x))} \quad \text{for } x\in[x_i,x_{i+1}], i\in\{1,\cdots,N-1\}, 
\end{equation}
where $m_N$ is the normalization constant to make $\widetilde{\rho}_N$ belong to $\mathcal{P}_{ac,p}([-l,l])$. One can check that $\widetilde{\rho}_N(t)$ narrowly convergent to $\rho(t)$ as $N\to\infty$ for all $t\in[0,T]$. By the monotonicity of $p_i$ and $\psi$, we have 
\begin{equation}\label{estimate of psi compo pi}
N\min\{\Delta x_i,\Delta x_{i+1}\}\leq \psi^{-1}(p_i(x))\leq N\max\{\Delta x_i,\Delta x_{i+1}\}\quad \text{for } x\in [x_i,x_{i+1}].
\end{equation}
By Lemma \ref{lem 6.6}, this yields 
\begin{equation*}
a_1\leq \psi^{-1}(p_i(x))\leq a_2\quad \text{for } x\in [x_i,x_{i+1}].
\end{equation*} 
Now the proof reduces to show that $\widetilde{\rho}_N$ gives a good estimate of $g_N(\mu_N)$ and $g(\rho)$, that is 
\begin{equation}\label{6.11}
    \liminf_{N\to\infty} g_N(\mu_N)\geq \liminf_{N\to\infty}g(\widetilde{\rho}_N)\geq g(\rho).
\end{equation}
where $(\widetilde{\rho}_N)_{N\geq2}$ is the sequence associated to $(\mu_N)_{N\geq 2}$ defined as in \eqref{tilde rho_N}. The second inequality above is due to $\widetilde{\rho}_N\rightharpoonup\rho$ narrowly and the narrow lower semi-continuity of $g$, see \cite[Corollary 2.4.10]{AGS}. Now we check the first inequality. Let us denote $\nu_N=m_N\widetilde{\rho}_N$. Noticing that $H^{\prime\prime}(x)=\frac{1}{x^3}h^{\prime\prime}\left(\frac{1}{x}\right)$, we have 
\begin{equation*}
\begin{aligned}
I(\nu_N)^p&=\int_{-l}^l |\nu_N^\prime(x)H^{\prime\prime}(\nu_N(x))|^p \nu_N(x)dx\\
&=\int_{-l}^l |\nu_N^\prime(x)|^p\left|h^{\prime\prime}(\frac{1}{\nu_N(x)})\right|^p \nu_N(x)^{1-3p}dx\\
&=\int_{-l}^l \left|\nu_N^\prime(x)\right|^p\left|\psi^\prime\left(\frac{1}{\nu_N(x)}\right)\right|^p \nu_N(x)^{1-3p} dx\\
&=\int_{-l}^l \left|\left(\psi\left(\frac{1}{\nu_N(x)}\right)\right)^\prime\right|^p \nu_N(x)^{1-p}dx\\
&=\sum_{i=1}^{N-1}\int_{x_i}^{x_{i+1}}\left|p_i^\prime(x)\right|^p\left|\psi^{-1}(p_i(x))\right|^{p-1}dx\\
&=\frac{1}{N^p}\sum_{i=1}^{N-1}\int_{x_i}^{x_{i+1}}\left|\frac{\psi_{i+1}-\psi_i}{\Delta x_{i+1}}\right|^p\left|\psi^{-1}(p_i(x))\right|^{p-1}dx
\end{aligned}
\end{equation*}
By $\eqref{estimate of psi compo pi}$, it holds
\begin{equation*}
I(\nu_N)^p\leq  \frac{1}{N}\sum_{i=1}^{N-1}\left|\psi_{i+1}-\psi_i\right|^p\left|\max\left\{1,\frac{\Delta x_i}{\Delta x_{i+1}}\right\}\right|^{p-1}. 
\end{equation*}
By Lemma $\ref{key lemma}$, we have $\frac{\Delta x_i}{\Delta x_{i+1}}\to 1$ as $N\to\infty$ uniformly for $i\in\{1,\cdots, N-1\}$. Thus for any $\epsilon>0$, there exists $N(\epsilon)$ large enough such that $\max\left\{1,\frac{\Delta x_i}{\Delta x_{i+1}}\right\}<1+\epsilon$ for all $N\geq N(\epsilon)$ and $i\in\{1,\cdots, N-1\}$. For such $N$ we obtain 
\begin{equation*}
    g(\nu_N)^p\leq \frac{1+\epsilon}{N}\sum_{i=1}^{N-1}\left|\psi_{i+1}-\psi_i\right|^p\leq (1+\epsilon)g_N(\mu_N)^p.
\end{equation*}
By taking the limits $N\to\infty$ and $\epsilon\to0$ in this order, we get 
\begin{equation}
    \liminf_{N\to\infty}g(\nu_N)\leq \liminf_{N\to\infty}g_N(\mu_N).
\end{equation}
In order to prove \eqref{6.11}, we only need to show that $\liminf\limits_{N\to\infty}g(\nu_N)\geq\liminf\limits_{N\to\infty}g(\widetilde\rho_N)$. Compute 
\begin{equation*}
\begin{aligned}
    g(\nu_N)^p=g(m_N\widetilde\rho_N)^p&=m_N^{p+1}\int_{-l}^l \widetilde{\rho}_N^\prime(x)^pH^{\prime\prime}(m_N\widetilde{\rho}_N(x))^p\widetilde{\rho}_N(x)dx\\
    &\geq m_N^{p+1}f(m_N)^p g(\widetilde{\rho}_N)^p,
\end{aligned}
\end{equation*}
where $f$ is as in Hypothesis \ref{H2}. Since $\widetilde{\rho}_N(t)\rightharpoonup\rho(t)$ narrowly as $N\to\infty$ for all $t\in[0,T]$, we have $m_N\to 1$ as $N\to\infty$. This completes the proof. 

\end{proof}

\section{$\Gamma$-convergence of the discrete energy.}

We show that the discrete energy $E_N$ is $\Gamma$-convergent to the continuum energy $E$ with respect to the $W_p$-distance. 
\begin{definition}[$\Gamma$-convergence]
We say that the discrete energy $(E_N)_{N \geq 2}$ is $\Gamma$-convergent to the continuum energy $E$ with respect to $W_p$-distance if the following two conditions hold for all $\rho \in \mathcal{P}_p(\Omega)$:
\begin{itemize}
    \item [(i)]("liminf" condition) All sequences $(\mu_N )_{N\geq2}$ with $\mu_N \in \mathcal{A}_N(\Omega)$ such that $W_p(\mu_N,\rho) \to 0$ as $N\to\infty$ satisfy $E(\rho)\leq \liminf_{N\to\infty} E_N(\mu_N)$.
    \item [(ii)]("limsup" condition) There exists a recovery sequence with respect to $W_p$ for $\rho$.
\end{itemize}
\end{definition}

To obtain the $\Gamma$-convergence, we require that $H$ satisfies the following additional condition: there exist continuous functions $f_1,f_2:[0,\infty) \to\mathbb R$ such that $f_1(1) = 1$ and $f_2(1) = 0$, and
\begin{equation}\label{5.1}
H(\alpha x) \leq f_1(\alpha)H(x) + f_2(\alpha)x \text{ for all } x,\alpha \in[0, \infty).   
\end{equation}
This is still satisfied by typical energy such as \eqref{typical energy}. 

\begin{theorem}
Let $H$ satisfy Hypothesis \ref{H1} and \eqref{5.1}. Then $(E_N)_{N\geq 2}$ $\Gamma$-converges to $E$.  
\end{theorem}

\begin{proof}
We follow the same strategy as in \cite{CPSW}. The "liminf" condition can be obtained from $(C2)$ proved in Proposition \ref{C2 holds}. To prove the "limsup" condition, we need to find a recovery sequence for any $\rho\in\mathcal{P}_p(\Omega)$ with respect to the $W_p$-distance. This is done by two steps. First, the recovery sequence is constructed for any $\rho\in \mathcal{G}(\Omega)$, and then relax this assumption on $\rho$ and prove the general result for any $\rho\in \mathcal{P}_{ac,p}(\Omega)$ by a density argument. Here we do the construction of the recovery sequence for $\rho\in\mathcal{G}(\Omega)$ as in \cite[Lemma 5.5]{CPSW}, which replies on the pseudo-inverse of the distribution function of $\rho$. And then by the same argument as in \cite[Lemma 6.6]{CPSW}, we can extend this result to any $\rho\in \mathcal{P}_{ac,p}(\Omega)$. We omit details here. 
\end{proof}


\begin{acknowledgement}
The author is supported by JSPS Grant-in-Aid for Transformative Research Areas(B) No. 23H03798.  The author would also like to thank Professor Jun Masamune for giving the support and encouragement.  
\end{acknowledgement}

\end{document}